\documentclass{amsart}
\usepackage{amsthm,amsfonts,amsmath,amscd,amssymb,latexsym,epsfig}
\usepackage{diagrams}

\newcommand{\cx}{{\mathbb C}}

\newcommand{\re}{\operatorname{Re}}
\newcommand{\im}{\operatorname{Im}}

\newcommand{\Ker}{\operatorname{Ker}}

\newcommand{\Hilb}{\operatorname{Hilb}}

\newcommand{\Gr}{\operatorname{Gr}}
\newcommand{\Mat}{\operatorname{Mat}}
\newcommand{\Rat}{\operatorname{Rat}}

\newcommand{\ol}{\overline}

\numberwithin{equation}{section}

\newtheorem{theorem}{Theorem}[section]

\newtheorem{lemma}[theorem]{Lemma}

\newtheorem{proposition}[theorem]{Proposition}

\theoremstyle{remark}

\newtheorem{remark}[theorem]{Remark}

\newtheorem{definition}[theorem]{Definition}

\newtheorem{example}[theorem]{Example}

\newcommand{\oH}{{\mathbb{H}}}

\newcommand{\oN}{{\mathbb{N}}}

\newcommand{\oP}{{\mathbb{P}}}

\newcommand{\oR}{{\mathbb{R}}}

\newcommand{\sM}{{\mathcal{M}}}   
\newcommand{\sN}{{\mathcal{N}}}
\newcommand{\sO}{{\mathcal{O}}}

\begin{document}

\title[Differential geometry of Hilbert schemes]{Differential geometry of Hilbert schemes of curves in a projective space}
\author{Roger Bielawski \& Carolin Peternell}
\address{Institut f\"ur Differentialgeometrie,
Leibniz Universit\"at Hannover,
Welfengarten 1, 30167 Hannover, Germany}
\email{bielawski@math.uni-hannover.de}

\thanks{Both authors are members of, and the second author is fully supported by the DFG Priority
Programme 2026 ``Geometry at infinity".}
\begin{abstract} We describe the natural geometry of Hilbert schemes of curves in $\oP^3$ and, in some cases, in $\oP^n$, $n\geq 4$.
\end{abstract}

\subjclass{53C26, 53C28, 14C05, 14H50}

\maketitle

\thispagestyle{empty}

It has been observed in \cite{Sigma} that the Hilbert scheme of real cohomologically stable curves of fixed genus and degree in $\oP^3$, not intersecting a fixed real line, carries a natural pseudo-hyperk\"ahler structure. This observation was made in a much more general context of curves in  twistor spaces of arbitrary hyperk\"ahler $4$-manifolds and relies on the isomorphism $\oP^3\backslash \oP^1\simeq \sO_{\oP^1}(1)^{\oplus 2}$. If, however, we want to describe the differential gemetry of all (real and stable) projective space curves with a fixed genus and degree, then having to remove a line from $\oP^3$  is clearly unsatisfactory, 
\par
\par
In the present article we describe such a natural differential-geometric structure on an open subset of the Hilbert scheme of real curves of degree $d$ and genus $g$ in $\oP^3$.
Rather than a  hypercomplex structure, which is a decomposition of the tangent bundle $T^\cx M$ as $E\otimes \cx^2$ for some quaternionic vector bundle $E$ (plus integrability conditions), the natural geometry of the real Hilbert scheme is what we call a quaternionic $4$-Kronecker structure, i.e. a bundle map $\alpha:E\otimes \cx^4\to T^\cx M$ for some quaternionic vector bundle (again plus integrability conditions). It turns out that these structures have a rich geometry, which is closely related to hypercomplex and quaternionic geometry. We also discuss the complex analogue of these structures, which is the geometry on an open subset of the full Hilbert scheme, i.e. not just real curves.
\par
We also consider Hilbert schemes of curves in $\oP^n$ for $n\geq 4$. It turns out, however, that we can expect open subsets with nontrivial geometry only for a very restricted range of $d$ and $g$. Nevertheless, such values do exist, e.g.  $g=0$ and any $d\geq n$.
\par
The article is organised as follows. In the next section we introduce abstract Kronecker structures on complex and real manifolds, their integrability and twistor spaces. In the second section we discuss their differential geometry and their relation with quaternionic and hypercomplex geometry. The following section is given to describing the natural integrable Kronecker structure on the Hilbert schemes of projective curves. In the final section we show that our point of view leads to new insights even for lines in $\oP^3$. 

\section{Kronecker module structures on manifolds}

An {\em $r$-Kronecker module} is a linear map $\alpha:V_0\otimes \cx^r\to V_1$, where $V_0$ and $V_1$ are finite-dimensional complex vector spaces. In other words $\alpha$ is a representation of a quiver with $2$ vertices $v_0,v_1$ and $r$ arrows from $v_0$ to $v_1$. A Kronecker module is called {\em quaternionic} if $r=2s$ is even, $V_0$ is equipped with a quaternionic structure $\sigma_0$ (i.e. $\dim V_0$ is also even), $V_1$ has a real structure $\tau$, and $\alpha$ satisfies $\alpha(\sigma_0(v)\otimes \sigma(z))=\tau\circ\alpha (v\otimes z)$, where $\sigma$ denotes the standard quaternionic structure on $\cx^{2s}$:
\begin{equation}\sigma(z_0,z_1,\dots,z_{2s-1})=(-\bar z_1,\bar z_0,-\bar z_3,\bar z_2,\dots, -\bar z_{2s-1},\bar z_{2s-2}).\label{sigma}\end{equation}
\begin{definition} Let $M$ be complex manifold. An {\em $r$-Kronecker structure of rank $k$} on $M$ consists of a  vector bundle $E$ of rank $k$ and a bundle map
$\alpha:E\otimes \cx^r\to T^M$ such that, for each $m\in M$, $\alpha_m$ is a Kronecker module and $\left.\alpha_m\right|_{E\otimes z}$ is injective for any  $z\in \cx^r\backslash\{0\}$.
\end{definition}
\begin{definition} Let $M$ be real manifold. A {\em quaternionic $r$-Kronecker structure of rank $k$} on $M$ consists of a quaternionic vector bundle $E$ of rank $k$ and a bundle map
$\alpha:E\otimes \cx^r\to T^\cx M$ such that, for each $m\in M$, $\alpha_m$ is a quaternionic Kronecker module and $\left.\alpha_m\right|_{E\otimes z}$ is injective for any  $z\in \cx^r\backslash\{0\}$.
\end{definition}
\begin{remark} If $M$ is real-analytic manifold with a real analytic quaternionic Kronecker structure, then we obtain a Kronecker structure on a complex thickening $M^\cx$ of $M$ by complexifying transition functions and the homomorphism $\alpha$. Conversely, if a complex manifold $M$ is equipped with an antiholomorphic involution $\tau$ and the Kronecker structure on $M$ is compatible with $\tau$, then the fixed-point set $M^\tau$ of $\tau$ has an induced quaternionic Kronecker structure.\label{complexify}
\end{remark}
\begin{remark} Consider the case of a quaternionic $2$-Kronecker structure with $k=\frac{1}{2}\dim M$. Whenever $\alpha_m$ is surjective, it induces an isomorphism $T_m M \simeq E\otimes \cx^2$ compatible with the quaternionic structure. Thus the submanifold of $M$, consisting of points where $\alpha_m$ is surjective, is an almost hypercomplex manifold. 
Even in this simplest case, new interesting examples arise by dropping the assumption that $T_m M \simeq E\otimes \cx^2$ everywhere. Thus we show in Example \ref{blowup} that $\oR {\rm P}^4$ has a natural $2$-Kronecker structure, smoothly extending the flat hypercomplex structure of $\oR^4$. For more results on the geometry of $2$-Kronecker structures with  $k=\frac{1}{2}\dim M$ see \cite{BP3}.
\end{remark}
\begin{remark} 
 Quaternionic $2$-Kronecker structures of arbitrary rank can be viewed as a particular case of {\em almost $\rho$-quaternionic} structures considered in   \cite{Pantilie}. They include (if $k>\frac{1}{2}\dim M$) the {\em generalised hypercomplex} structures of \cite{BielTAMS}. 
\end{remark}
\begin{remark} If $\alpha$ is an isomorphism  (in particular $kr=\dim M$), then such a Kronecker structure is an {\em almost Grassmann} structure considered in \cite{AG}.
\end{remark}
\begin{remark} In addition to quaternionic Kronecker structures, one can also consider {\em split quaternionic}  Kronecker structures, where the involution $\sigma$ is replaced by $\sigma(z_0,z_1,\dots,z_{2s-1})=(\bar z_1,\bar z_0,\dots, \bar z_{2s-1},\bar z_{2s-2})$. Even more generally, we can put the minus sign in front of some $\bar z_{2i-1}$, but not all.
\label{split}
\end{remark}
Observe that, for any line $v\in \oP^{r-1}$, the restriction of $\alpha$ to $E\otimes v$  defines a rank $k$ subbundle $T^vM$ of $T M$.
\begin{definition} An $r$-Kronecker structure on a (complex or smooth) manifold $M$ is called {\em integrable} if the subbundle $T^v M$ is involutive for each $v\in \oP^{r-1}$ (i.e. $\bigl[T^v M,T^v M\bigr]\subset T^v M$).\label{integrable}
\end{definition}
\begin{remark} Let $r^\prime <r $. For any $r^\prime$-dimensional subspace $W$ of $\cx^r$ we can restrict $\alpha$ to $E\otimes W$ and obtain an $r^\prime$-Kronecker structure of the same rank. These structures are parametrised by  $\Gr_{r^\prime}(\cx^r)$. 
\label{many}\end{remark}
\begin{remark} We can relax the assumption that $\alpha$ is injective on each $E\otimes z$ as follows. Let $r^\prime <r $. Then $\alpha:E\otimes \cx^{r}\to TM$ is called a {\em weak $(r,r^\prime)$-Kronecker structure}, if for any  $W\in \Gr_{r^\prime}(\cx^r)$, the set
$$M_W=\{m\in M; \left.\alpha_m\right|_{E\otimes z}\enskip \text{is injective for each $z\in W$}\}$$
is open and dense in $M$ and $\bigcup M_W=M$. In particular each $M_W$ has a genuine $r^\prime$-Kronecker structure.
We shall say that such a weak Kronecker structure is integrable if all these $r^\prime$-Kronecker structures are integrable. There is an analogous definition of weak quaternionic Kronecker structures. 
\label{generalised}\end{remark}

\subsection{Twistor spaces}

Let $M$ be a complex  manifold equipped with an integrable $r$-Kronecker structure of rank $k$. We have an integrable holomorphic distribution $D$ of rank $k$ on $M\times \oP^{r-1}$, given by $D|_{M\times {[v]}}=\alpha(E\otimes v)$.
\begin{definition} An integrable Kronecker structure is called {\em regular} if  the foliation determined by $D$ is simple, i.e. the space of its leaves is a manifold. This manifold (of dimension $\dim M+r-k-1$) is then called the {\em twistor space} of $(M,E,\alpha)$.
\end{definition}
The twistor space is equipped with a natural holomorphic submersion $\pi:Z\to \oP^{r-1}$, and any element $m\in M$ defines a section of $\pi$.
If we start with a real-analytic integrable quaternionic Kronecker structure on a real-analytic manifold $M$, then we can proceed as in Remark \ref{complexify} and obtain a Kronecker structure on a complex thickening $M^\cx$ of $M$. If this complexified Kronecker structure is regular, then we obtain the twistor space $Z=Z(M^\cx)$ of $M^\cx$ which is equipped, in addition, with a real structure $\tau$ covering the real structure $\sigma$ on $\oP^{2s-1}$. This twistor space obviously depends on the choice of complex thickening. In many cases there exists a minimal twistor space $Z$, i.e. the inverse limit of twistor spaces $Z(U)$ over the directed poset consisting of open neighbourhoods of $M$ in some complexification $M^\cx$ such that the above foliation is simple on $U$.  The question whether this inverse limit exists and whether it is a complex manifold,  is an interesting topological problem which we shall not investigate here. In the natural examples which interest us, the twistor
  space is given, so that we obtain $M$ as the manifold of real sections.

\medskip

Let  $\pi:Z\to \oP^{r-1}$ be the twistor space of a regular integrable Kronecker structure and denote by $\hat m$ the section of $\pi$ corresponding to a point $m\in M$.
The definition of $Z$ implies that the normal bundle $N$ of $\hat m$ in $Z$ is given 
by the following exact sequence of sheaves on $\oP^{r-1}$:
\begin{equation} 0\to E_m\otimes \sO(-1)\stackrel{\alpha_m}{\longrightarrow} T_mM\otimes \sO\longrightarrow N\to 0.\label{NK}
\end{equation}
It follows that $H^1(\hat m,N_{\hat m/Z})=0$ and $H^0(\hat m,N_{\hat m/Z})\simeq T_mM$ and so $M$ can be recovered as a (component of) Kodaira moduli space of embedded $\oP^{r-1}$-s in $Z$.
\begin{remark} \eqref{NK} shows that the normal bundles of sections of the twistor projection are {\em Steiner bundles} (cf. \cite{DK,Jardim}). 
\end{remark}

\begin{remark} As observed in Remark \ref{many}, any  subspace $W$ of $\cx^r$ induces an (integrable) $\dim W$-Kronecker structure on $M$.
Its twistor space is easily seen to be $\pi^{-1}\bigl(\oP(W)\bigr)\subset Z.$ On the other hand suppose that $M$ is equipped with an integrable weak $(r,r^\prime)$-Kronecker structure as defined in Remark \ref{generalised}. Suppose also that all induced $r^\prime$-Kronecker structures are regular, i.e. for each $W\in \Gr_{r^\prime}(\cx^r)$ we obtain a corresponding twistor space $Z_W$ of the corresponding $M_W$. It is easy to see that these $Z_W$ combine to give again complex manifold $Z$ with a holomorphic submersion $\pi:Z\to \oP^{r-1}$ such that $Z_W=\pi^{-1}(\oP(W))$. It is no longer true, however, that all (or even any) points of $M$ correspond to sections of $\pi$. 
\label{restrict}
\end{remark}

Let us now prove the converse of the above construction.
\begin{theorem} Let $Z$ be a complex manifold with a surjective holomorphic submersion $\pi:Z\to \oP^{r-1}$. Then, for each $k\in \oN$, the family of sections  of $\pi$, the normal bundle $N$ of which admits a resolution of the form
\begin{equation} 0\to \sO(-1)^{\oplus k}\to\sO^{\oplus n}\to N\to 0,\label{Steiner}\end{equation}
is a smooth manifold of dimension $n$ with a natural regular integrable $r$-Kronecker structure of rank $k$.
\label{twistor}\end{theorem}
\begin{proof} The resolution \eqref{Steiner} implies that $h^1(N)=0$ and $h^0(N)=n$. Thus the sections with such a resolution belong to a smooth Kodaira moduli space of dimension $n$. Moreover, the property of having a resolution of this form is open \cite[Corollary 3.3]{DK}, so that we do obtain a complex manifold $M$ of dimension $n$ of sections with resolution \eqref{Steiner}. We have a double fibration
\begin{equation}M\stackrel{\tau}{\longleftarrow} M\times \oP^{r-1}\stackrel{\eta}{\longrightarrow} Z,\label{double}\end{equation}
where $\eta(x,v)=t_x(v)$, $t_x:\oP^{r-1}\to Z$ being the section corresponding to $x\in M$.
Furthermore, the existence of resolution \eqref{Steiner} implies that $N$ is globally generated and that the kernel of the natural surjective map $ H^0(N)\otimes \sO\to N$ is of the form $V_0\otimes \sO(-1)$ for a vector space $V_0$ of dimension $k$. Tensoring \eqref{Steiner} with $\sO(-r+1)$ and taking the long exact on cohomology shows that $V_0$ is canonically isomorphic to $H^{r-2}(N(-r+1))$. On the other hand, the normal bundle of each section  is isomorphic to the restriction of the vertical tangent bundle $T_\pi Z=\Ker d\pi$ to the section. Therefore the higher direct image sheaf $\tau_\ast^{r-2}\eta^\ast T_\pi Z(-r+1)$ is a rank $k$ complex vector bundle $E$ on $X$, and we have a canonical short exact sequence at each $m\in M$
$$ 0\to E_m\otimes \sO(-1)\stackrel{A}{\longrightarrow} T_mM\otimes \sO\longrightarrow N\to 0.$$
We obtain a canonically defined bundle map $\alpha:E\otimes \cx^{r}\to TM$ by setting
$\left.\alpha\right|_{E_m\otimes z}=A|_{[z]}$. It follows immediately that $\left.\alpha\right|_{E_m\otimes z}$ is injective for every $z$.
\par
Thus we obtain a canonical $r$-Kronecker structure of rank $k$ on $M$ and it remains to show that it is integrable. Let $v\in \oP^{r-1}$. The bundle $T^vM=\alpha(E\otimes v)$ is the kernel of the evaluation map $TM\to N_v$. This is the same as the kernel of the map $d\eta$ in \eqref{double} restricted to $v\in \oP^{r-1}$ and therefore integrable.
\end{proof}
\begin{remark} If $r=2s$ and $Z$ isequipped with a real structure $\tau:Z\to Z$ covering the real structure \eqref{sigma} on $\oP^{2s-1}$, then the space of real sections with resolution as in the theorem carries a quaternionic $2s$-Kronecker structure (here $k$ must be even). This follows immediately from the above proof, since $\sO_{\oP^{2s-1}}$ and   $\sO_{\oP^{2s-1}}(-1)$ have, respectively, canonical real and quaternionic structures. Thus \eqref{Steiner} implies that, over $M^\tau$, $E$ has an induced quaternionic structure, so that $\alpha$ is a quaternionic Kronecker module.\end{remark}
\begin{remark} It follows from the proof that the constructions of the above theorem and of the twistor space of a regular integrable Kronecker structure are indeed converse to each other, with the caveat that if we start with $Z$ as above, construct $(M,E,\alpha)$ and then its twistor space $Z(M)$, then $Z$ does not have to coincide with  $Z(M)$. All we can say in general is that there exists a local biholomorphism $\rho:Z(M)\to Z$, which makes the following diagram commute:
\begin{equation*}
\begin{diagram} M\times \oP^{r} & \rTo & Z(M)\\
& \rdTo & \dTo_\rho \\
& & Z
\end{diagram}
\end{equation*}
\end{remark}

\begin{example} Let $Z$ be $\oP^3$ blown up in a real line $l$. This blow-up can be viewed as making all planes containing $l$ disjoint and so we have a natural projection $\pi:Z\to \oP^1\simeq l^\ast=\{L\in (\oP^3)^\ast;\; l\subset \oP(L)\}$. We also have a real structure $\tau$ on $Z$ obtained from the real structures of $\oP^3$ and of $\oP^1$. The exceptional divisor is $E\simeq \oP(N_{l/\oP^3})\simeq l\times \oP^1$ is $\tau$-invariant and its normal bundle is isomorphic to $ \sO(1,-1)$. $Z\backslash E$ is just $\oP^3 \backslash l$, which, together with the projection $\pi$, is the twistor space of the flat $\oR^4$. Thus any section of $\pi$, which is contained in $Z\backslash E$ has normal bundle $\sO(1)\oplus\sO(1)$. On the other hand any real section $s$ of $\pi$ which meets $E$ in a point $x$ must also meet it in $\tau(x)\neq x$.
This means that its projection $\bar s$ in $\oP^3$ meets $l$ in two distinct points and, since the degree of $\bar s$ is $1$, $\bar s=l$. Thus any real section meeting $E$ is entirely contained in $E$. As a line on $E\simeq \oP^1\times\oP^1$, it has bidegree $(1,1)$, so its normal bundle in $E$ is $\sO(2)$. Combining with $N_{E/Z}\simeq \sO(1,-1)$, we conclude that the normal bundle of such a section in $Z$ is $\sO(2)\oplus \sO$. In both cases the normal bundle $N$ of a section has a resolution of the form
$$ 0\to \sO(-1)\oplus\sO(-1)\to \sO^{\oplus 4}\to N\to 0,$$
so the real sections form a $4$-dimensional manifold $M^4$ with a quaternionic $2$-Kronecker structure. As observed above, real sections not meeting $E$ form $\oR^4$, while the remaining sections are real curves of degree $(1,1)$ on $\oP^1\times\oP^1$, so these form $\oR {\rm P}^3$. It follows that $M^4\simeq \oR^4\cup \oR {\rm P}^3\simeq \oR {\rm P}^4$.
\par
In order to identify the Kronecker structure we describe sections explicitly. Choose $l$ to be $\{[z_0,z_1,0,0]\in \oP^3\}$. We can then identify $Z$ with
$$ \bigl\{([z_0,z_1,z_2,z_3],[x_0,x_1]\in \oP^3\times \oP^1;\; z_2x_0+z_3x_1=0\bigr\},$$
and $\pi$ is the projection onto the second factor. It follows that sections of $\pi$ are of the form
$$ [x_0,x_1]\mapsto \bigl([a_0x_0+a_1x_1,b_0x_0+b_1x_1,-cx_1,cx_0],[x_0,x_1]\bigr),$$
and hence the space $X^4$ of sections is
$$\bigl\{[a_0,a_1,b_0,b_1,c]\in \oP^4\;; \;c=0\implies a_0b_1-a_1b_0\neq 0\bigr\}.$$
The real curves satisfy, in addition, $b_0=-\bar a_1$, $b_1=\bar a_0$, $c\in \oR$, so that the manifold of real sections is indeed $\oR{\rm P}^4$. The fibre of the bundle $E$ at a section $x\in X$ consists of sections of $N(-1)$ and the map $\alpha$ is the natural multiplication $H^0(N(-1))\otimes H^0(\sO_{\oP^1}(1))\to H^0(N)$. Thus the image of each $H^0(N(-1))\otimes L$, $L\in H^0(\sO_{\oP^1}(1))$, in $H^0(N)$ consists of infinitesimal deformations of $x$ which vanish  at the intersection points of $x$ with $\oP(L)$.  It follows  that $E$ is the restriction of $\sO_{\oP^4}(1)\oplus \sO_{\oP^4}(1)$ to $X$
and the map $\alpha:E\otimes \cx^2\to T\oP^4$ is the restriction of the Euler sequence projection
$$0\to \sO_{\oP^4}\to \sO_{\oP^4}(1)^{\oplus 5}\to T\oP^4\to 0$$
to $$\bigl(\sO_{X}(1)\oplus \sO_{X}(1)\bigr)\otimes \cx^2\simeq \sO_{X}(1)^{\oplus 4}=\sO_{X}(1)^{\oplus 4}\oplus 0\subset \sO_{\oP^4}(1)^{\oplus 5}.$$
\label{blowup}\end{example}

\section{Differential geometry of integrable Kronecker structures}

\subsection{Ward transform}
Let $\alpha:E\otimes \cx^r\to TM$ be a regular integrable Kronecker structure on a complex manifold $M$, and let $Z$ be the corresponding twistor space.\\
Consider the double fibration \eqref{double} and write $Y=M\times \oP^{r-1}$. We interested in the sheaf $\Omega^\ast_\eta$ of $\eta$-vertical forms on $Y$, i.e. the exterior algebra of $\Omega^1(Y)/\eta^\ast\Omega^1(Z)$. It is a locally free sheaf and the corresponding vector bundle $T^\ast Y/\eta^\ast T^\ast Z$ is dual to $TY/\Ker d\eta$. The construction of the twistor space, given in the previous section, shows that $TY/\Ker d\eta$ restricted to $\{m\}\times \oP^{r-1}$ is isomorphic to $E_m\otimes \sO(-1)$ and, consequently, the direct image sheaf $\tau_\ast \Omega^1_\eta$ is isomorphic to $(E\otimes \cx^r)^\ast $.  Recall (e.g. from \cite{BE}) that there is a first order differential operator $d_\eta:\Omega^0(Y)\to \Omega^1_\eta$ obtained by composing the exterior derivative with the projection onto $\Omega^1_\eta$. We can identify the push-forward of $d_\eta$ as follows:
\begin{lemma} The operator $\tau_\ast d_\eta:\Omega^0 M\to (E\otimes \cx^r)^\ast$ is equal to $\alpha^\ast\circ d$.\end{lemma}
\begin{proof} 
$\Omega^1_\eta$ fits into an exact sequence:
\begin{equation}0\to \eta^\ast T^\ast_\pi Z\to \tau^\ast \Omega^1 M\to \Omega^1_\eta\to 0.\label{eta}\end{equation}
Its restriction to $\{m\}\times \oP^{r-1}$ is the dual of the sequence \eqref{NK}, i.e.
$$ 0\to N^\ast\longrightarrow T_m^\ast M\otimes \sO\stackrel{\alpha^\ast}{\longrightarrow} E_m^\ast\otimes \sO(1)\to 0.$$
Taking the push-forward proves the statement.\end{proof}

We now want to discuss the Ward transform for $M$. Let  $F$ be an $M$-uniform holomorphic vector bundle on $Z$, i.e. $h^0(\eta(\tau^{-1}(m),F)$ is independent of $m$. We then obtain a holomorphic vector bundle $\hat F=\tau_\ast\eta^\ast F$ on $M$. There exists a relative flat connection $\nabla_\eta$ on $\eta^\ast F$ and its pushforward to $M$ is a first-order differential operator $D:\hat F\to \tau_\ast \Omega^1_\eta(F)=\tau_\ast(\Omega^1_\eta\otimes \eta^\ast F)$. Tensoring \eqref{eta} with $F$ and restricting to $\{m\}\times \oP^{r-1}$ gives
\begin{equation} 0\to N^\ast\otimes \eta^\ast F\longrightarrow T_m^\ast M\otimes \eta^\ast F\stackrel{\alpha^\ast}{\longrightarrow} E_m^\ast\otimes \eta^\ast F(1)\to 0.\label{eta2}\end{equation}
Thus  $\tau_\ast \Omega^1_\eta(F)\simeq E^\ast\otimes \widehat{F(1)}$, and the operator $D:\hat F\to E^\ast\otimes \widehat{F(1)}$ satisfies the following ``Leibniz rule":
$$D(fs)=\sigma(df\otimes s)+fDs,$$
where $\sigma:T^\ast M\otimes \hat F\to  E^\ast\otimes \widehat{F(1)}$ is a bundle homomorphism given by the composition
$$ T^\ast M\otimes \hat F\stackrel{}{\longrightarrow} E^\ast\otimes (\cx^r)^\ast\otimes \hat F\longrightarrow E^\ast\otimes\widehat{F(1)},$$
where the first map is $\alpha^\ast\otimes 1$ and the second map is the multiplication of sections $ H^0(\sO(1))\otimes\hat F\to \widehat{F(1)}$.

\begin{remark} $\sigma$ is nothing else but the principal symbol of the operator $D$. If we write $\alpha=(\alpha_1,\dots,\alpha_r)$, where each $\alpha_i:E\to TX$, and similarly write the multiplication map as $(\beta_1,\dots,\beta_r)$, where each $\beta_i:\hat F\to \widehat{F(1)}$ (and where we used the same basis of $\cx^r$), then $\sigma=\alpha_1^\ast\otimes\beta_1+\dots +\alpha_r^\ast\otimes\beta_r$.
\end{remark}

 Obviously if we start with a quaternionic Kronecker structure on a real manifold $M$ and $F$ is equipped with a compatible real structure, then we obtain such an operator on the corresponding real vector bundle over $M$. We also recall that the flatness of the relative  connection $\nabla_\eta$ means that holomorphic sections of $F$ yield solutions of $Ds=0$.

\subsection{Quaternionic Kronecker structures with $k=\frac{1}{2}\dim M$\label{half}} Integrable quaternionic Kronecker structures with rank equal to $\frac{1}{2}\dim M$ are closely related to hypercomplex geometry. Since the map $\alpha$ is equivariant with respect to the quaternionic structures on $E$ and $\cx^r$ and the complex conjugation on $T^\cx M$, it follows that
 $\ol{T^v M}=T^{\sigma(v)}M$ for any $v\in\oP^{r-1}$. Therefore 
each $v\in \oP^{r-1}$ defines an integrable complex structure $I_v$ on an open subset $M_v$ of $M$ where $T^v_m\cap T_m^{\sigma(v)}=0$ by setting
$T^{0,1}_mM=T^v_m$,  $T^{1,0}_mM=T_m^{\sigma(v)}$ ($M_v$ may be empty). Observe also that a choice of $v$ determines $\sigma(v)$, which in turn determines  a $\sigma$-invariant $2$-dimensional subspace $W$ of $\cx^r$. Restricting $\alpha$ to $E\otimes W$ establishes an isomorphism $T^\cx M_v\simeq E\otimes \cx^2$ and shows that $M_v$ is a hypercomplex manifold. Moreover $M_v=M_{v^\prime}$ for any $v^\prime \subset W$. Thus we have a family of hypercomplex structures, parametrised by real lines in $\oP^{r-1}$, i.e. by $\oH{\rm P}^{s-1}$ ($r=2s$),
but each of them defined only on an open subset $M_q$,  $q\in \oH{\rm P}^{s-1}$. 
Of course what is defined on all of $M$ is the quaternionic $2$-Kronecker structure determine by $W$ (cf. Remark \ref{many}).

We can define  the following manifold parametrising points of $M$ {\em and} the hypercomplex structures at each point:
\begin{equation} \tilde M=\{(m,q)\in M\times \oH{\rm P}^{s-1}; \; m\in M_q\}.\label{Mhat}\end{equation}
It turns out that there is a natural quaternionic structure on $\tilde M$, the restriction of which to each $M_q$ is the corresponding hypercomplex structure.
 Consider namely the twistor space $\pi:Z\to \oP^{r-1}$ of $(M,E,\alpha)$ and a real section $\hat m$ of $\pi$ corresponding to $m\in M$. Let $l\simeq \oP^1$ be a line lying on $\hat m$. The normal bundle of such a $\oP^1$ fits into the exact sequence
\begin{equation} 0\to N_{l/\hat m}\to N_{l/Z}\to \left.N_{\hat m/Z}\right|_l\to 0.\label{seq0}\end{equation}
Since $N_{l/\hat m}\simeq \sO(1)^{r-2}$ and the restriction of the sequence \eqref{NK} shows that $\left.N_{\hat m/Z}\right|_l$ splits as the direct sum of line bundles with nonnegative degrees, we conclude that $H^1(l, N_{l/Z})=0$ and $\dim H^0(l, N_{l/Z}))=\dim M+2r-4$. It follows that the parameter space of $\tau$-invariant projective lines lying on some $\hat m$ is $M\times \oH{\rm P}^{s-1}$. The points of $\tilde M$ correspond precisely to those lines $l$ for which $\left.N_{\hat m/Z}\right|_l\simeq \sO(1)^k$. It follows then that  $N_{l/Z}\simeq \sO(1)^{k+r-2}$ so that $\tilde M$ coincides with the parameter space of $\tau$-invariant projective lines in $Z$ with normal bundle splitting as a sum of $\sO(1)$. Therefore $\tilde M$ has a natural quaternionic structure such that each $M_q$ is a quaternionic submanifold and the restriction of the quaternionic structure of $\tilde M$ to each $M_q$ is the corresponding hypercomplex structure.

\section{Kronecker structures on Hilbert schemes of curves in $\oP^{n}$}

\subsection{$4$-Kronecker structures on Hilbert schemes of curves in $\oP^3$}

We consider the Hilbert scheme $\Hilb_{d,g}$ of closed subschemes of $\oP^3$ with Hilbert polynomial $h(m)=dm-g+1$ and its open subscheme $M_{d,g}$  consisting of all 
$C\in \Hilb_{d,g}$ staisfying the following two conditions:
\begin{itemize}
\item[1)] $h^1(C,\sN_{C/\oP^3}(-1))=0$,
\item[2)] $C$ has no planar components, i.e. the sheaf map $\sO_C(-1)\stackrel{\cdot t}{\longrightarrow} \sO_C$ is injective for any $t\in H^0(\sO_{\oP^3}(1))$. 
\end{itemize}
Let us make some observations. First of all, condition 2) together with the Hilbert polynomial implies that $C\cap H$ is a $0$-dimensional scheme of length $d$ for any hyperplane $H\subset \oP^3$. Choose now a line $\oP^1$ in $\oP^3$ disjoint from $C$. Then the projection $C\to\oP^1$ is finite-to-one and all its fibres have the same length. Thus $C$ is (locally) Cohen-Macaulay. The first condition implies that $h^1(C,\sN_{C/\oP^3})=0$ and therefore $M_{d,g}$ is smooth, since the codimension of $C$ is $2$ \cite[Cor. 8.5]{Hart}. Its tangent space at each $C$ is identified with $H^0(C,\sN_{C/\oP^3})$. 
\begin{remark} There are no curves satisfying 2) with $(d,g)=(1,0)$ or $(d,g)=(2,0)$. For all other values of $(d,g)$, condition 2) is satisfied by all smooth nonplanar space curves of degree $d$ and genus $g$. On the other hand, condition 1) is satisfied by a general smooth curve if $d-3\geq\frac{3g+1}{4}$  \cite[Thm. II.3.4]{Hirsch}. Thus, at least in this range $M_{d,g}$ is nonempty (hence of dimension $4d$).
\end{remark}
We shall now show that $M_{d,g}$ (if nonempty) has a natural regular integrable $4$-Kronecker structure of rank $2d$, which restricts to a quaternionic $4$-Kronecker structure on its $\sigma$ invariant part $M_{d,g}^\sigma$, where $\sigma$ is the antiholomorphic involution \eqref{sigma} on $\oP^3$.
\begin{remark} $M_{d,g}^\sigma$ can be empty, e.g. for $g=0$ and $d$ even. In this case one can consider instead curves invariant under the other involution defined in Remark \ref{split}. The submanifold of such curves  will have a split quaternionic $4$-Kronecker structure.\label{empty}
\end{remark}
It follows from condition 1) and from the fact that the normal sheaf $\sN_{C/\oP^3}$ of a Cohen-Macaulay curve in $\oP^3$ is torsion-free \cite[Cor. 3.2]{BP2} that, for each $t\in H^0(\sO_{\oP^3}(1))$, we have a short exact sequence
\begin{equation} 0\to \sN_{C/\oP^3}(-1)\stackrel{\cdot t}{\longrightarrow} \sN_{C/\oP^3}\longrightarrow \left.{\sN_{C/\oP^3}}\right|_{C\cap H}\to 0,\label{Nseq}\end{equation}
 where $H=\oP(\Ker t)$. Since the ideal of $C\cap H$ in $H$ is $J_C\otimes \sO_H$, it follows that $\sN_{C\cap H/H}\simeq \left.{\sN_{C/\oP^3}}\right|_{C\cap H}$. Thus $h^0(C\cap H, \left.{\sN_{C/\oP^3}}\right|_{C\cap H})=h^0(C\cap H, \sN_{C\cap H/H})$, but the latter is equal to $2d$ since it is the dimension of the tangent space at $C\cap H$ to  the Hilbert scheme of $d$ points in $H$, which is smooth. Thus we conclude from \eqref{Nseq}  that $h^0(N_{C/\oP^3}(-1))=2d$. 
We define a rank $2d$ holomorphic vector bundle $E$ on $M_{d,g}$ by setting $E_C=H^0(C,\sN_{C/\oP^3}(-1))$. Taking the long exact sequence of \eqref{Nseq} defines a bundle map
$$ \alpha:E\otimes H^0(\sO_{\oP^3}(1))\to TM_{d,g},$$
which is injective on each $E\otimes v$, i.e. $\alpha$ is a $4$-Kronecker structure. If  $C$ is $\sigma$-invariant, then its normal sheaf has a natural real structure and, consequently, $\sN_{C/\oP^3}(-1)$ has a natural quaternionic structure. It follows that $\left.E\right|_{M_{d,g}^\sigma}$ is a quaternionic vector bundle and $\alpha|_{M_{d,g}^\sigma}$ is a quaternionic $4$-Kronecker structure.
\par
The subbundle $T^v M_{d,g}=\alpha(E\otimes v)$ of $TM_{d,g}$ is just the kernel of the evaluation map $H^0(C, \sN_{C/\oP^3})\to H^0(C\cap H,\sN_{C\cap H/\oP^3})$ and hence involutive: the leaf of the distribution $T^v M_{d,g}$ consists of deformations of $C$ leaving $C\cap H$ fixed. Therefore our Kronecker structure on $M_{d,g}$ is integrable. To show that it is regular, we shall construct a complex manifold $Z_d$ from which $M_{d,g}$ arises as in Theorem \ref{twistor}.
 
\subsection{Twistor space}

Let $Q\simeq \oP(T^\ast\oP^3)$ be the incidence variety of hyperplanes in $\oP^3$:
$$ Q=\{(p,L)\in \oP^3\times {(\oP^3)}^\ast\;;\enskip p\in L\},$$
where we view elements of ${(\oP^3)}^\ast$ as planes in $\oP^3$. For each integer $d\geq 1$ we define a variety $Z_d$ as the relative Hilbert scheme of $d$ points with respect to the projection $Q\to {(\oP^3)}^\ast$. Thus $Z_d$ consists of all planar $0$-dimensional schemes of length $d $ in $\oP^3$. It is a smooth projective manifold of dimension $2d+3$ equipped with a natural fibration $\pi:Z_d\to {(\oP^3)}^\ast$ with fibres  isomorphic to ${(\oP^2)}^{[d]}$.
\par
The real structure $\sigma$ on $\oP^3$ induces a real structure on ${(\oP^3)}^\ast$ which in turn induces real structures on $Q$ and on $Z_d$. 
We denote the real structure on $Z_d$ by $\tau$.
\par
Any element $C$ of $M_{d,g}$  defines a section of the projection $\pi:Z_d\to {(\oP^3)}^\ast$, which we denote by $\hat C$. We denote the normal sheaf of $C$ in $\oP^3$ by $\sN$ and the normal bundle of $\hat C\simeq \oP^3$ in $Z_d$ by $\hat N$. Since $\hat N$ is isomorphic to the vertical bundle $T_\pi Z_d$ restricted to $C$, the fibre of $\hat N$ at each $D\in Z_d$ is canonically isomorphic to $H^0(D,\sN)$ (cf. \cite[Lemma 4.1]{BP2}).
The sequence \eqref{Nseq} implies that the evaluation map $H^0(C,\sN)\otimes \sO_{\oP^3}\to \hat N$ is surjective and its kernel is 
$H^0(C,\sN(-1))\otimes \sO_{\oP^3}(-1)$. Thus the normal bundle of each section $\hat C$, $C\in M_{d,g}$, has a resolution of form \eqref{Steiner}, and so all assumptions of Theorem \ref{twistor} are satisfied. We recover $M_{d,g}$ with its $4$-Kronecker structure from $Z_d$ as an open submanifold of the parameter 
space of embedded $\oP^3$-s in $Z_d$ with Steiner normal bundle. The $\tau$-invariant sections correspond to points of $M_{d,g}^\sigma$. 
\begin{remark} If we replace condition 2) in the definition of $M_{d,g}$ with ``$C$ is pure-dimensional and Cohen-Macaulay", then we 
obtain a weak $(4,2)$-Kronecker structure (as defined in Remark \ref{generalised}) on the manifold of all such $C$. In particular, for any hyperplane $H\subset \oP^3$, we obtain a $2$-Kronecker structure on the open subset of such $C$ which do not have a component contained in $H$. The twistor space of this weak Kronecker structure (as defined in Remark \ref{restrict}) is still $Z_d$. \label{extend}\end{remark}

Let $C\in M_{d,g}$ and let $\hat N$ be the normal bundle of the corresponding $\hat C\simeq \oP^3\subset Z_d$. We want to describe the generic splitting type of $\hat N$. The restriction of $\hat N$ to any line $l\in {(\oP^3)}^\ast$ also has the resolution of the form \eqref{Steiner}. Suppose that $\hat N|_l$ has a direct summand of the form $\sO(k)$ with $k>1$, and consequently there exists a section of $\hat N|_l$ vanishing at $k$ distinct points. This means that there is a corresponding section $s$ of $\sN$ on $C$ which vanishes at the intersections $D_1,\dots,D_k$ of $C$ with $k$ distinct planes in $l$. Thus $s$ is a section of $\sN[-D_1-\dots-D_k]$. If $C$ does not intersect the line $l^\ast=\bigcap\{L; L\in l\}$, then the divisors $D_1,\dots,D_k$ are disjoint, and, consequently, such an $s$ corresponds to a section of $\sN(-k)$. Thus we can describe the generic splitting type of $\hat N$ as:
\begin{proposition} Let $C\in M_{d,g}$ and let $l \subset {(\oP^3)}^\ast$ be a line such that $C\cap\bigcap\{L; L\in l\}=\emptyset$. Then 
$$ \hat N|_{l}\simeq  \sO^{r_0}\oplus\bigoplus_{i\in\oN} \sO(i)^{r_i},$$
where $r_i=\dim H^0(C,\sN(-i))-(i+1)\sum_{j>i}r_j$. 
\end{proposition}
In particular, if $H^0(C,\sN(-2))=0$, then the generic splitting type of $\hat N$ is $ \sO(1)^{2d}$. On the other hand if $C\in M_{4,1}$, then its normal bundle is $\sO_C(2)\oplus \sO_C(2)$, and so the generic splitting type of $\hat N$ is $\sO(2)^{\oplus 2}\oplus \sO(1)^{\oplus 4}\oplus \sO^{\oplus 2}$.

\subsection{Rational curves}

We can be more explicit about this Kronecker structure in the case $g=0$. For any $d\geq 3$ we consider non-planar immersed rational curves $C$ of degree $d$. Here ``immersed" is used in the differential-geometric sense, i.e. $C$ is given as the image of a degree $d$ rational map
\begin{equation} \phi:\oP^1\to \oP^3,\label{phi}\end{equation}
the differential of which is everywhere injective. Such a curve is l.c.i. and, owing to results of Ghione and Sacchiero \cite{GS}, its normal bundle splits as $\sO(d+a)\oplus\sO(d+b)$ where $a,b\geq 2$ and $a+b=2d-2$. In particular we have $H^1(C,N_{C/\oP^3}(-1))=0$, so such curves belong to $M_{d,0}$. We denote by $\Rat_d$ the subset of $M_{d,0}$ consisting of such curves.  As a manifold $\Rat_d=P_d/GL(2,\cx)$, where $P_d$ is an open subset of quadruples of homogeneous polynomials of degree $d$ in two variables. For such a quadruple $\phi(x_0,x_1)=\bigl(\phi_0(x_0,x_1),\phi_1(x_0,x_1),\phi_2(x_0,x_1),\phi_3(x_0,x_1)\bigr)$ denote by $D\phi$ its Jacobian matrix $\bigl(\partial\phi_i/\partial x_j\bigr)$. Then, as in \cite[Lemma 1.1]{GS}, we have an exact sequence of sheaves on $\oP^1$:
\begin{equation} 0\to\sO(1)^{\oplus 2}\stackrel{D\phi}{\longrightarrow} \sO(d)^{\oplus 4}\longrightarrow N\to 0,\label{Jac}\end{equation}
where $N=\phi^\ast N_{C/\oP^3}$.
We denote by $\mu$ the projection from $TP_d$ onto $T\Rat_d$, i.e. the map induced on global sections by \eqref{Jac}.
From the defintion of a Kronecker structure, the map $\alpha$ sends $E\otimes t$ to sections of $N$ vanishing on $H= \oP(\Ker t)$. If $t=\sum_{i=0}^3 t_iz_i$, then $C\cap H$ is the image under $\phi$ of the zero set $\Lambda_t$ of $\sum_{i=0}^3t_i\phi_i(x_0,x_1)$. Suppose for the moment that $\Lambda_t$ consists of $d$ distinct points $\lambda_1,\dots,\lambda_q$. Let $s\in\alpha(E\otimes t)$ and write $s=\mu(q)$ where $q=(q_0,q_1,q_2,q_3)$ is a quadruple of degree $d$ polynomials. Since $s$ vanishes on $\Lambda_t$,
$q$ must be in the image of $D\phi$ at points of $\Lambda_t$. Choose an arbitrary $S_i$ in the image of $D\phi$ at each $\lambda_i$, $i=1,\dots,d$.
There exist $d$ vectors $(u_i,v_i)\in \cx^2$ such that $S_i=D\phi(u_i,v_i)(\lambda_i)$. Let $p_1(x_0,x_1)$ and $p_2(x_0,x_1)$ be degree $d-1$ homogenous polynomials with $p_1(\lambda_i)=u_i$, $p_2(\lambda_i)=v_i$, $i=1,\dots,d$, and set
$$ q^\prime=(q_1^\prime,q_2^\prime,q_3^\prime,q_4^\prime)= D\phi (p_1,p_2)\mod \sum_{i=0}^3t_i\phi_i.$$
Then $q^\prime(\lambda_i)=S_i$. Any other quadruple $q$ of polynomials of degree $d$ with the same values at the $\lambda_i$ differs from $q^\prime$ by  $u\sum_{i=0}^3t_i\phi_i$, where $u\in\cx^4$.
Moreover, observe from \eqref{Jac}, that $\mu$ vanishes on the image of linear polynomials. Therefore we may assume that $x_1^2$ divides both $p_1$ and $p_2$ (i.e. $p_1$ and $p_2$ have zero constant and linear terms when written in the affine coordinate $x_1/x_0$). This gives the following description of $E$ and $\alpha$: $E$ is the trivial bundle with fibre $\cx^{2d}$ which we write as $E^\prime\oplus \cx^4$, where $E^\prime$ is the vector space of pairs of homogeneous polynomials $p_1(x_0,x_1),p_2(x_0,x_1)$ of degree $d-1$ divisible by $x_1^2$, and the map $\alpha$ is given by (apriori only for $t$ such that $\Lambda_t$ consists of distinct points, but the formula obviously extends to all $t$):
$$ \alpha\bigl(((p_1,p_2)\oplus u)\otimes t\bigr)=\mu\left(\left(D\phi(p_1,p_2) \mod \sum_{i=0}^3t_i\phi_i\right)+u\sum_{i=0}^3t_i\phi_i\right).$$
\begin{remark} As shown in \S\ref{half}, the quaternionic Kronecker structure on $\sM_{d,g}^\sigma$ induces a  hypercomplex structure on the submanifold $(\sM_{d,g}^\sigma)_W$ for each $\sigma$-invariant subspace $W$ of $\cx^4$, i.e. for each real line $l$ in $\oP^3$. It is easy to see that $(\sM_{d,g}^\sigma)_W$ consists  of real curves avoiding the line $l$. This is the hypercomplex structure introduced in \cite{Sigma}, and so it is actually pseudo-hyperk\"ahler. As observed in Remark \ref{empty}, $\sM_{d,g}^\sigma$ may be empty, but there always is a complexified hypercomplex structure (i.e. an integrable action of $\Mat_2(\cx)$ on the tangent bundle) on the submanifold of all curves in $M_{d,g}$ such that the restriction of $\alpha$ to $E\otimes W$ is an isomorphism (this submanifold may, however, be empty for every $W$, e.g. on $M_{4,1}$).                                                                                                                                              
                                    
 The main result of \cite{BP2} is that for $g=0$ this hypercomplex structure is always flat.
\end{remark}

\subsection{Curves in $\oP^n$, $n\geq 4$}

Let $\Hilb_{d,g,n}$ denote the Hilbert scheme of closed subschemes of $\oP^n$ with Hilbert polynomial $h(m)=dm-g+1$. We can try and define  $ M_{d,g,n}$ analogously to the case $n=3$. However, the condition $h^1(C,\sN_{C/\oP^n}(-1))=0$ imposes now strong restrictions on $d$ and $g$.  Indeed, we can easily compute the degree of $\sN_{C/\oP^n}(-1)$ for a smooth (or just l.c.i) curve from the normal sequence  and obtain $\deg \sN_{C/\oP^n}(-1)=2d+2g-2$, and then, from the Riemann-Roch theorem, $\chi(\sN_{C/\oP^n}(-1))=2d -(n-3)(g-1)$. Therefore, if $h^1(C,\sN_{C/\oP^n}(-1))=0$, then $2d\geq (n-3)(g-1)$. A further restriction is that we cannot include now all Cohen-Macaulay curves, since for $n\geq 4$ the condition $h^1(C,\sN_{C/\oP^4})=0$ is not sufficient for the smoothness of $ M_{d,g,n}$. We have to restrict ourselves to  l.c.i.\ curves. With these modifications, however, we do obtain an $(n+1)$-Kronecker structure on  $ M_{d,g,n}$:
\begin{proposition} Assume that $2d \geq (n-3)(g-1)$ and define $ M_{d,g,n}$ as the open subscheme of $\Hilb_{d,g,n}$ consisting of all $C\in \Hilb_{d,g,n}$ which are l.c.i.\ and satisfy conditions 1) and 2) of the definition of $ M_{d,g}$. If $ M_{d,g,n}$ is nonempty, then it is a smooth manifold of dimension $(n+1)d-(n-3)(g-1)$ equipped with a natural regular integrable $(n+1)$-Kronecker structure of rank $2d-(n-3)(g-1)$.
\end{proposition}
\begin{proof} The dimension of $ M_{d,g,n}$ is computed from $\chi(\sN_{C/\oP^n})$ in the same way as for $\chi(\sN_{C/\oP^n}(-1))$ above. Now all arguments and constructions of the preceding subsection go through, except that we need to define the twistor space $Z_{d,n}$ as consisting of l.c.i.\ $0$-dimensional subschemes lying on hyperplanes in $\oP^n$ (this guarantees that $Z_{d,n}$ is smooth).
\end{proof}
\begin{remark}
Kronecker structures of small rank $k$ are, in a sense, degenerate (as an extreme case consider $k=0$). ``Nondegeneracy" should probaly mean that the map $\alpha$ is generically surjective. This implies that $kr\geq \dim M$, which in our case translates into  the following inequality on $g$ and $d$:
$$ (n+1)d\geq n(n-3)(g-1).$$
There do exist values of $(d,g,n)$ in this range for which $ M_{d,g,n}$ is nonempty. For example, a nondegenerate immersed rational curve always satisfies $h^1(C,\sN_{C/\oP^n}(-1))=0$, so that $ M_{d,0,n}$ with $d\geq n$ is a complex manifold of dimension $dn+d+n-3$ equipped with a natural  regular integrable $(n+1)$-Kronecker structure of rank $2d+n-3$.
\end{remark}

\section{The weak Kronecker structure on  $S^4$}

As pointed out in Remark \ref{extend}, $\Hilb_{1,0}$, i.e. the manifold of lines in $\oP^3$, has a weak $(4,2)$-Kronecker structure, the twistor space of which is $Z_1\simeq \oP(T^\ast\oP^3)$.
Of course $\Hilb_{1,0}=\Gr_2(\cx^4)$. The bundle $E$ on $\Gr_2(\cx^4)$ coincides with the tautological bundle $S$ and the homomorphism $\alpha$ is given by
$$ \alpha(s\otimes z)=s\otimes (z+S)\in S\otimes \bigl(\cx^4/S\bigr)\simeq T\Gr_2(\cx^4).$$
Thus $\alpha=0$ at points $(H,z)\in\Gr_2(\cx^4)\times \cx^4$ such that $z\in H$.
\par
We now want to discuss the induced quaternionic Kronecker structure on real lines, i.e. on $\Gr_2(\cx^4)^\sigma=S^4$.
Recall (Remark \ref{generalised})) that any  $W\in \Gr_1(\oH^2)\simeq  S^4$ defines a $2$-Kronecker structure on the corresponding $M_W$. In the present case $M_W=S^4\backslash \{W\}$ and the corresponding $2$-Kronecker structure is simply the flat hypercomplex structure on $\oR^4\simeq S^4\backslash\{W\}$. In particular the manifold $\tilde M$, defined in \eqref{Mhat} as parametrising points and hypercomplex structures, is just $(S^4\times S^4)\backslash \Delta$. As observed in \S\ref{half}, $\tilde M$ carries a natural quaternionic structure, which we now proceed to identify (recall that the product of two non-flat quaternionic manifolds is usually no longer quaternionic, so this is not any sort of product quaternionic structure). In order to this we need to consider real lines in $Z_1$ with normal bundle $\sO(1)^{\oplus 4}$. We shall in fact consider all real lines in $Z_1$, which will provide a natural compactification of $\tilde M$.
 \par
 Recall that $Z_1= \oP(T^\ast\oP^3)$, which we identify with a  quadric hypersurface in $\oP^3\times\oP^3$:
$$ Q=\{([x],[y])\in \oP^3\times \oP^3\;;\sum_{i=0}^3 x_iy_i=0\}.$$
We consider lines in $\oP^3\times \oP^3$ which are contained in $Q$, i.e.
\begin{equation}  \oP^1\ni \zeta\mapsto \bigl([a+b\zeta ],[c+d\zeta]\bigr),\quad a\cdot c=b\cdot d=a\cdot d+b\cdot c=0.\label{lines}\end{equation}
The normal bundle of such a line $l$ fits into the exact sequence
\begin{equation} 0\to N_{l/Q}\to N_{l/\oP^3\times \oP^3}\to \left.N_{Q/\oP^3\times \oP^3}\right|_l\to 0.\label{seq1}\end{equation}
The normal bundle of $Q$ in $\oP^3\times \oP^3$ is $\sO(1,1)$ and hence $\left.N_{Q/\oP^3\times \oP^3}\right|_l\simeq \sO(2).$ On the other hand, $l$ is a curve of bidegree $(1,1)$ on 
$l_1\times l_2\in \oP^3\times \oP^3$, where $l_1=\{[a+b\zeta ];\zeta\in \oP^1\}$, $l_2=\{[c+d\zeta ];\zeta\in \oP^1\}$. Thus $N_{l/l_1\times l_2}\simeq \sO(2)$ and, since
$N_{l_1\times l_2/\oP^3\times \oP^3}\simeq \sO(1,0)^{\oplus 2}\oplus\sO(0,1)^{\oplus 2}$, it follows that $N_{l/\oP^3\times \oP^3}\simeq \sO(2)\oplus \sO(1)^{\oplus 4}$. Hence \eqref{seq1} becomes
$$ 0\to N_{l/Q}\to \sO(2)\oplus \sO(1)^{\oplus 4}\to \sO(2)\to 0,$$
and so $N_{l/Q}$ is either $\sO(1)^{\oplus 4}$ or $\sO(2)\oplus\sO\oplus \sO(1)^{\oplus 2}$ with the latter occuring precisely when $l_1\times l_2\subset Q$, i.e.
\begin{equation}N_{l/Q}\simeq\begin{cases} \sO(2)\oplus\sO\oplus \sO(1)^{\oplus 2} &\text{if $a\cdot d=b\cdot c=0$,}\\ \sO(1)^{\oplus 4} & \text{otherwise.}\end{cases}\label{O(2)}
\end{equation}

\subsection{Real curves}
If we equip $Q$ with the antiholomorphic involution 
$$ ([x],[y])\mapsto ([-\bar y],[\bar x]),$$
then $Q$ becomes the twistor space of the quaternionic manifold $\Gr_2(\cx^4)$. The corresponding family of real lines is given by 
$$ \oP^1\ni \zeta\mapsto \bigl([x-\zeta\bar y],[y+\zeta \bar x]\bigr)\in Q, \enskip \text{where}\enskip|x|^2=|y|^2=1, \; \sum_{i=0}^3 x_iy_i=0,$$
modulo the action of $U(2)$ on $\oP^1$, i.e.
$$ \begin{pmatrix} x \\ y\end{pmatrix}\mapsto A\begin{pmatrix} x \\ y\end{pmatrix}.$$
Observe that \eqref{O(2)} implies that the normal bundle of such a curve is always $\sO(1)^4$.


The real structure induced on $Q$ from $\oP^3$ is a different one, namely: 
\begin{equation} \tau([x],[y])=([\sigma(x)],[\sigma(y)]), \label{real}
\end{equation}
where $\sigma(z_0,z_1,z_2,z_3)=(-\bar z_1,\bar z_0,-\bar z_3,\bar z_2)$.
The corresponding family of  real sections is given by
$$ \oP^1\ni \zeta\mapsto \bigl([x+\zeta\sigma(x)],[y+\zeta \sigma(y)]\bigr)\in Q, \enskip \text{where}\enskip x\cdot y=0,\enskip x\cdot\sigma(y)+\sigma(x)\cdot y=0,$$
modulo the action induced by the action of $GL(1,\oH)\subset GL(2,\cx)$ on $\oP^1$. According to \eqref{O(2)}, the normal bundle of such a curve splits as $\sO(1)^4$, unless, in addition, $x\cdot\sigma(y)-\sigma(x)\cdot y=0$.\\
We denote by $X$ the manifold of all $\tau$-invariant lines in $Q$, by $X^{\rm o}$ the open submanifold of lines with $N_{l/Q}\simeq \sO(1)^4$ and write $X_\infty=X\backslash X^{\rm o}$.  The double fibration $\oP^3\leftarrow Q\rightarrow\oP^3$ induces a double fibration $S^4\leftarrow X\rightarrow S^4$. We have
\begin{proposition}
\begin{itemize}
\item[(i)] $X$ is the real (non-oriented) blow-up of $S^4\times S^4$ in the antidiagonal $\{(x,-x)\}$ and $X_\infty$ is the exceptional divisor of the blow-up, i.e.  $X_\infty\simeq \oP\bigl({\rm T}^\ast S^4\bigr)$.
\item[(ii)]  With respect to either of the projections $X\to S^4$, $X$ is an $\oR{\rm P}^4$ bundle over $S^4$. More precisely, $X=\oP(T^\ast S^4\oplus \sO_{S^4})$, where $\sO_{S^4}$ is the trivial line bundle $S^4\times \oR$.
\end{itemize}
\end{proposition}
\begin{proof}
If we write $q_0=x_0+x_1j$, $q_1=x_2+x_3 j$, $p_0=y_0-jy_1$, $p_1=y_2-jy_3$ (all of them elements of $\oH)$, then the above conditions on $x,y$ can be written simply as \begin{equation} \im_\oH(q_0p_0+q_1p_1)=0,\label{imH}\end{equation} and the action of $GL(1,\oH)\simeq \oH^\ast$ is given by 
\begin{equation} (q_0,q_1,p_0,p_1)\mapsto (uq_0,uq_1,p_0u^{-1},p_1u^{-1}), \quad u\in \oH^\ast.\label{Hacts}\end{equation} The condition on a curve to have the normal bundle isomorphic to $\sO(2)\oplus\sO\oplus \sO(1)^2$ is $x\cdot\sigma(y)-\sigma(x)\cdot y=0$, which means that the real part of $q_0p_0+q_1p_1$ vanishes as well.  Thus
$$X=\{\bigl((q_0,q_1),(p_0,p_1)\bigr)\in \oH^2\backslash\{0\}\times  \oH^2\backslash\{0\} ; q_0p_0+q_1p_1\in\oR\}/GL(1,\oH)\times \oR^\ast,$$
where $GL(1,\oH)$ acts as above and and $\oR^\ast$ acts by diagonal multiplication. Similarly
$$X_\infty=\{\bigl((q_0,q_1),(p_0,p_1)\bigr)\in \oH^2\backslash\{0\}\times  \oH^2\backslash\{0\} ; q_0p_0+q_1p_1=0\}/GL(1,\oH)\times \oR^\ast.$$
The double fibration $S^4\leftarrow X\rightarrow S^4$ is given by $(q_0,q_1,p_0,p_1)\mapsto (q_1^{-1}q_0,p_1p_0^{-1})\in \oH {\rm P}^{1}\times \oH {\rm P}^{1}$.
Observe that $X_\infty$ maps to the antidiagonal. 
Moreover,  the quotient of $\{\bigl((q_0,q_1),(p_0,p_1)\bigr)\in \oH^2\backslash\{0\}\times  \oH^2 ; q_0p_0+q_1p_1=0\}$ by $GL(1,\oH)$ is ${\rm T}^\ast \oH P^{1}$ and, hence, 
$X_\infty\simeq \oP\bigl({\rm T}^\ast \oH P^{1}\bigr)$. Consider, on the other hand, the fibre $F$ over a point away from the antidiagonal in $\oH {\rm P}^{1}\times \oH {\rm P}^{1}$. If $(q_0,q_1,p_0,p_1)$ represents a point of $F$, then $q_0p_0+q_1p_1=q_1(q_1^{-1}q_0+p_1p_0^{-1})p_0$ is real and nonzero. If $(\tilde q_0,\tilde q_1,\tilde p_0,\tilde p_1)$ represents another point of $F$, then 
$$ \tilde q_1(\tilde q_1^{-1}\tilde q_0+\tilde p_1\tilde p_0^{-1})\tilde p_0=\tilde q_1(q_1^{-1}q_0+p_1p_0^{-1})\tilde p_0$$
is again real and nonzero, and hence, $\tilde q_1q_1^{-1}=r\tilde p_0^{-1}p_0$ for a nonzero real number $r$. It follows that $(q_0,q_1,p_0,p_1)$ and $(\tilde q_0,\tilde q_1,\tilde p_0,\tilde p_1)$  belong to the same orbit of  $GL(1,\oH)\times \oR^\ast$, so that $F$ is a point. This proves (i).
\par
We now prove (ii) for the projection onto the second $S^4$, i.e. $([q_0,q_1],[p_0,p_1])\mapsto p_1p_0^{-1}$. For each $[p_0,p_1]\in S^4$ the equation \eqref{imH} is a triple of linearly independent linear equations in $\oR^8$ and so it defines a rank $5$ subbundle $E$ of the trivial bundle $S^4\times \oR^8$. We define a rank $4$ subbundle $E^\prime$ of $E$ by setting $\re(q_0p_0+q_1p_1)=0$. Thus $E^\prime$ is a subbundle of trivial bundle $S^4\times \oR^8$ defined by the equation $q_0p_0+q_1p_1=0$, i.e. $E^\prime\simeq T^\ast S^4$. The quotient bundle $E/E^\prime$ is then a real line bundle, hence trivial. Finally, since any extension of smooth vector bundles splits, $E\simeq T^\ast S^4\oplus \sO_{S^4}$.
\end{proof}

\subsection{The quaternionic structure}
As shown in \S\ref{half}, $X\backslash X_\infty\simeq S^4\times S^4\backslash\{(x,-x)\}$ has a natural quaternionic structure, which we proceed to identify.
\par
Consider $\oH^2\oplus\oH^2\simeq \cx^4\oplus\cx^4$ with complex (for the complex structure $i$) coordinates $x_0,\dots,x_3,y_0\dots,y_3$ and a flat pseudo-hyperk\"ahler metric of signature $(8,8)$:
\begin{equation} g=\re\left(dx_1d\bar y_0-dx_0d\bar y_1+dx_3d\bar y_2-dx_2d\bar y_3\right).\label{metric}\end{equation}
The equations $x\cdot y=0$, $ x\cdot\sigma(y)+\sigma(x)\cdot y=0$ are  the moment maps equations for the $S^1$-action given by:
$$(x_0,\dots,x_3,y_0,\dots,y_3)\mapsto\bigl(e^{i\theta}x_0,e^{-i\theta}x_1,e^{i\theta}x_2,e^{-i\theta}x_3,e^{-i\theta}y_0,e^{i\theta}y_1,e^{-i\theta}y_2,e^{i\theta}y_3 \bigr).$$
Observe that that the length of the vector field $X$ generated by this action is equal to $x\cdot \sigma(y)$, so that the $\oH$-subspace $\oH X=\langle X, IX, JX,KX\rangle $ of the tangent space is nondegenerate (with respect to $g$) precisely on 
$$ U=\{\bigl((q_0,q_1),(p_0,p_1)\bigr)\in \oH^2\times  \oH^2 ;\enskip \re(q_0p_0+q_1p_1)\neq 0\}.$$
Observe also that the $\oH^\ast$-action given by \eqref{Hacts} generates at each point of $U$ a quaternionic subspace of signature opposite to $\oH X$ (i.e. $(0,4)$ if the latter is $(4,0)$ and vice versa). It follows that the quaternionic structure on $X^{\rm o}\simeq (S^4\times S^4)\backslash\{(x,-x)\}$\ is actually a pseudo-quaternion-K\"ahler metric of signature $(4,4)$ obtained as a quaternion-K\"ahler quotient by $S^1$ of the following pseudo-quaternion-K\"ahler manifold
$$ \{\bigl((q_0,q_1),(p_0,p_1)\bigr)\in \oH^2\times  \oH^2 ;\enskip \re(q_0p_0+q_1p_1)\neq 0\}/\oH^\ast,$$
where $\oH^\ast$ acts as in \eqref{Hacts}.

\end{document}